\documentclass[11pt]{article}

  \usepackage[colorlinks=true,linkcolor=blue,urlcolor=black,citecolor=blue,bookmarksopen=true]{hyperref}
  \usepackage{bookmark}

  \usepackage[bottom]{footmisc}

  \usepackage{amsthm,stmaryrd,amsfonts,amssymb,mathtools} 
  \usepackage[alphabetic]{amsrefs}

  \usepackage{tikz}
  \usepackage{tikz-cd}

  \usepackage{bbm} 

  \usepackage[textsize=scriptsize]{todonotes}

  \usepackage{enumitem} 

  \theoremstyle{plain}
  \newtheorem{theorem}{Theorem}[subsection]
  \newtheorem{proposition}[theorem]{Proposition}
  \newtheorem{corollary}[theorem]{Corollary}
  \newtheorem{lemma}[theorem]{Lemma}
  \newtheorem{remark}[theorem]{Remark}

  \theoremstyle{definition}
  \newtheorem{definition}[theorem]{Definition}
  \newtheorem{example}[theorem]{Example}

    \tikzset{
        symbol/.style={%
            draw=none,
            every to/.append style={%
                edge node={node [sloped, allow upside down, auto=false]{$#1$}}}
        }
    }

    \tikzset{pb/.style={"\lrcorner", very near start}} 
	\tikzset{pf/.style={"\llcorner", very near start}} 

    \newcommand{\fm}[1]{\textbf{#1}} 
    \renewcommand{\sf}[1]{\textsf{#1}}

    \newcommand{\cat}[1]{\mathsf{#1}}

    \newcommand{\sSet}{\cat{sSet}}
    \newcommand{\Cat}{\cat{Cat}}
    \newcommand{\sCat}{\cat{sCat}}
    
    \newcommand{\RelCat}{\cat{RelCat}}
    \newcommand{\MonCat}{\cat{MonCat}}

    \newcommand{\Hom}{\cat{Hom}}
    \newcommand{\Ob}{\cat{Ob}}

    \newcommand{\coCart}{\cat{coCart}}
    \newcommand{\sMonCat}{\cat{Mon(sCat)}}
    \newcommand{\opFib}{\cat{opFib}}

    \newcommand{\Gr}{\cat{Gr}}

    \newcommand{\op}{\mathrm{op}}

    \newcommand{\CC}{\mathfrak{C}} 
    \newcommand{\N}{\mathrm{N}}

    \newcommand{\RR}{\mathbb{R}}

    \newcommand{\St}{\sf{St}}
    \newcommand{\Un}{\sf{Un}}

    \newcommand{\var}[1]{\mathcal{#1}}

    \newcommand{\C}{\var{C}} 
    \newcommand{\D}{\var{D}}
    \newcommand{\E}{\var{E}}   
    \newcommand{\F}{\var{F}} 
    \newcommand{\K}{\var{K}}
    
    \newcommand{\V}{\var{V}}

    \renewcommand{\O}{\var{O}}

    \newcommand{\xto}[1]{\xrightarrow{#1}}
    
    \newcommand{\uq}{\mathrm{u.q.}}

    \newcommand{\rev}{\mathsf{rev}}

    \newcommand{\1}{\mathbf{1}}

    \newcommand{\bead}[1]{{\langle #1 \rangle}}

\begin{document}

\title{The Operadic Nerve, Relative Nerve, \\ and the Grothendieck Construction }
\author{Jonathan Beardsley and Liang Ze Wong}
\maketitle

\begin{abstract}
	We relate the relative nerve $\N_f(\D)$ of a diagram of simplicial sets $f \colon \D \to \sSet$ with the Grothendieck construction $\Gr F$ of a simplicial functor $F \colon \D \to \sCat$ in the case where $f = \N F$.
	We further show that any strict monoidal simplicial category $\C$ gives rise to a functor $\C^\bullet \colon \Delta^\op \to \sCat$, and that the relative nerve of $\N \C^\bullet$ is the operadic nerve $\N^\otimes(\C)$.
	Finally, we show that all the above constructions commute with appropriately defined opposite functors.
\end{abstract}

\tableofcontents

\section{Introduction}
	Given a simplicial colored operad $\O$, \cite{ha}*{2.1.1} introduces the \emph{operadic nerve} $\N^\otimes(\O)$ to be the nerve of a certain simplicial category $\O^\otimes$.
	This has a canonical fibration $\N^\otimes(\O) \to \N(\F in_*)$ to the nerve of the category of finite pointed sets which describes the $\infty$-operad associated to $\O$. 

	A special case of the above arises when one attempts to produce the \textit{underlying monoidal $\infty$-category} of a simplicial monoidal category $\C$. Following the constructions of \cite{dag2}*{1.6} and \cite{ha}*{4.1.7.17}, one first forms a simplicial category $\C^\otimes$ from a monoidal simplicial category $\C$, then takes its nerve to get $\N^\otimes(\C) := \N(\C^\otimes)$. We call this the \emph{operadic nerve of $\C$}, where the monoidal structure of $\C$ will always be clear from context. To be more precise, we should call this construction the operadic nerve of the underlying non-symmetric simplicial colored operad, or simplicial multicategory, of $\C$, but for ease of reading we do not. The above construction ensures that there is a canonical \textit{coCartesian} fibration $\N^\otimes(\C) \to \N(\Delta^\op)$, which imbues $\N(\C)$ with the structure of a \emph{monoidal $\infty$-category} in the sense of \cite{dag2}*{1.1.2}. Given that \cite{dag2} exists only in preprint form, we also refer the reader to \cite{gephaugenriched}*{\S 3.1} for a published (and more general than we will need) account of the operadic nerve of a simplicial multicategory. 

	Our paper is motivated by the following:
	if $\C$ is a monoidal fibrant simplicial category, then so is its opposite $\C^\op$. 
	We thus get a monoidal $\infty$-category $\N^\otimes(\C^\op)$.
	However, we could also have started with $\N^\otimes(\C)$ and arrived at another monoidal $\infty$-category $\N^\otimes(\C)_\op$ by taking `fiberwise opposites'. 

	We show that $\N^\otimes(\C^\op)$ and $\N^\otimes(\C)_\op$ are equivalent in the $\infty$-category of monoidal $\infty$-categories i.e.~that \emph{taking the operadic nerve of a simplicial monoidal category commutes with taking opposites} (Theorem \ref{thm:opcommute}). This follows from a more general statement about the relationship between the simplicial nerve functor, the enriched Grothendieck construction of \cite{beardswong}, and taking opposites (Theorem \ref{thm:F-op-commute}). In the process of proving the above, we also give a simplified description of the somewhat complicated \textit{relative nerve} of \cite{htt} (Theorem \ref{thm:gr-rel-nerve}) that we hope will be useful to others. 
	
	One corollary of our Theorem \ref{thm:opcommute} is the fact that  \textit{coalgebras} in the monoidal quasicategory $N^\otimes(\C)$ can be identified with the nerve of the simplicial category of \textit{strict} coalgebras in $\C$ itself, and that this relationship lifts to categories of comodules over coalgebras as well (this corollary and its implications are left to future work). There is well developed machinery in \cite{ha} for passing algebras and their modules from simplicial categories to their underlying quasicategories, but this machinery fails to work for coalgebras and comodules. As such, it is our hope that the work contained herein may lead, in the long run, to a better understanding of \textit{derived coalgebra}. 


	\subsection{Outline}

	We begin in a more general context:~in \S \ref{sec:rel-nerve-gr}, we review the relative nerve $\N_f(\D)$ of a functor $f \colon \D \to \sSet$ and the Grothendieck construction $\Gr F$ of a functor $F \colon \D \to \sCat$. 
	We show that when $F$ takes values in locally Kan simplicial categories, so that the composite $f \colon \D \xrightarrow{F} \sCat \xrightarrow{\N} \sSet$ takes values in quasicategories, we have an isomorphism associated to a commutative diagram: 
	\[ 
		\begin{aligned} \N(\Gr F) \cong \N_f(\D), \end{aligned}
		\quad \quad \quad \quad 
		\begin{aligned} 
			\begin{tikzcd} 
				\sCat^\D \ar[r, "\N \circ -"] \ar[d, "\Gr"'] & \sSet^\D \ar[d, "\N_{(-)}(\D)"]
				\\ 
				\opFib_{/\D}\ar[r, "\N"] & \coCart_{/\N(\D)}. 
			\end{tikzcd} 
		\end{aligned} 
	\]
	The relative nerve is itself equivalent to the $\infty$-categorical Grothendieck construction $\Gr_\infty \colon (\Cat_\infty)^{\N(\D)} \to \coCart_{/\N(\D)}$, yielding an equivalence of coCartesian fibrations
	\[
		\N(\Gr F) \simeq \Gr_\infty(\N(f)).
	\]

	In \S \ref{sec:monoid-struct}, we show that a strict monoidal simplicial category $\C$ gives rise to a functor $\C^\bullet \colon \Delta^\op \to \sCat$ whose value at $[n]$ is $\C^{n}$.
	We show that $\Gr\, \C^\bullet \cong \C^\otimes$, and thus that the operadic nerve $\N^\otimes(\C) := \N(\C^\otimes)$ factors as:
	\[
		\begin{tikzcd}
			\sMonCat \ar[r, "(-)^\bullet"] \ar[rr, bend right = 15, "(-)^\otimes"']  & \sCat^{\Delta^\op} \ar[r, "\Gr"] & \opFib_{/\Delta^\op} \ar[r, "\N"] & \coCart_{/\N(\Delta^\op)}
		\end{tikzcd}
	\]

	In \S \ref{sec:op-func}, we show that the above constructions interact well with taking  opposites, in that the following diagram `commutes:'
	\[
		\begin{tikzcd}[row sep = large]
			\sMonCat \ar[r, "(-)^\bullet"] \ar[d, "\op" description] & \sCat^{\Delta^\op} \ar[r, "\Gr"] \ar[d, "\op" description] & \opFib_{/\Delta^\op} \ar[r, "\N"]  
			& \coCart_{/\N(\Delta^\op)} \ar[d, "\op" description]
			\\
			\sMonCat \ar[r, "(-)^\bullet"] & \sCat^{\Delta^\op} \ar[r, "\Gr"] & \opFib_{/\Delta^\op} \ar[r, "\N"] & \coCart_{/\N(\Delta^\op)}		
		\end{tikzcd}
	\]
	We write `commutes' because we only check it \emph{on objects}, and only \emph{up to equivalence} in the quasicategory $\coCart_{/\N(\Delta^\op)}$.
	We conclude that $\N^\otimes(\C^\op)$ and the fiberwise opposite $\N^\otimes(\C)_\op$ are equivalent in the $\infty$-category of monoidal $\infty$-categories.

\section{Notation}

In large part, our notation follows that of Lurie's seminal works in higher category theory \cite{ha,htt}. However, here we point out certain notational conventions that may not be immediately obvious to the reader. Some of these conventions may be non-standard, but we adhere to them for the sake of precision.

\begin{enumerate}
	\item We will mostly avoid using the term ``$\infty$-category'' in any situation where a more precise term (e.g. quasicategory or simplicially enriched category) is applicable. We make one exception when we discuss the ``$\infty$-categorical'' Grothendieck construction of \cite{htt}.
	\item A special class of simplicially enriched categories are those in which all mapping objects are not just simplicial sets, but Kan complexes. We will refer to a simplicially enriched category with this property as ``locally Kan.'' 
	\item We will often use the term ``simplicial category'' to refer to a simplicially enriched category. There is no chance for confusion here because at no point do we consider simplicial object in the category of categories.
\end{enumerate}

\section{The relative nerve and the Grothendieck construction} 
\label{sec:rel-nerve-gr}
	The $\infty$-categorical Grothendieck construction is the equivalence
	\[
		\Gr_\infty \colon (\Cat_\infty)^S \xrightarrow{\quad \simeq \quad} \coCart_{/S}
	\]
	induced by the unstraightening functor $\Un^+_S \colon (\sSet^+)^{\CC[S]} \to (\sSet^+)_{/S}$ of \cite{htt}*{3.2.1.6}.
	Here, $\Cat_\infty$ is the quasicategory of small quasicategories, and $\coCart_{/S}$ is the quasicategory of coCartesian fibrations over $S \in \sSet$, and these are defined as nerves of certain simplicial categories.
	(See \ref{sec:models} and \ref{sec:st-un}, or \cite{htt}*{Ch. 3} for details.)

	In general, it is not easy to describe $\Gr_\infty \varphi$ for an arbitrary morphism $\varphi \colon S \to \Cat_\infty$.
	However, when $S$ is the nerve of a small category $\D$, and $\varphi$ is the nerve of a functor $f \colon \D \to \sSet$ such that each $fd$ is a quasicategory, the \emph{relative nerve} $\N_f(\D)$ of \cite{htt}*{3.2.5.2} yields a coCartesian fibration equivalent to $\Gr_\infty \N(f)$.

	If $f$ further factors as $\D \xrightarrow{F} \sCat \xrightarrow{\N} \sSet$, where each $Fd$ is a locally Kan simplicial category, we may instead form the simplicially-enriched Grothendieck construction $\Gr F$ and take its nerve.
	The purpose of this section is to show that we have an isomorphism of coCartesian fibrations
	\[
		\N(\Gr F) \cong \N_f(\D),
	\]
	thus yielding an alternative description of $\Gr_\infty \N(f)$.

	\subsection{The relative nerve $\N_f(\D)$}

	\begin{definition}[\cite{htt}*{3.2.5.2}]\label{def:relnerve}
		Let $\D$ be a category, and $f \colon \D \to \sSet$ a functor.
		The \fm{nerve of $\D$ relative to $f$} is the simplicial set $\N_f(\D)$ whose $n$-simplices are sets consisting of:
		\begin{enumerate}[label = (\roman*)]
			\item a functor $d \colon [n] \to \D$; write $d_i$ for $d(i)$ and $d_{ij} \colon d_i \to d_j$ for the image of the unique map $i \leq j$ in $[n]$,
			\item for every nonempty subposet $J \subseteq [n]$ with maximal element $j$, a map $s^J \colon \Delta^J \to fd_j$,
			\item such that for nonempty subsets $I \subseteq J \subseteq [n]$ with respective maximal elements $i \leq j$, the following diagram commutes:
			\begin{equation} 
			\label{eq:rel-nerve}
				\begin{tikzcd}
					\Delta^I \ar[r, "s^I"]  \ar[d, hookrightarrow] & fd_i \ar[d, "f d_{ij}"]
					\\
					\Delta^J \ar[r, "s^J"] & fd_j
				\end{tikzcd}
			\end{equation}
		\end{enumerate}
	\end{definition}

	For any $f$, there is a canonical map $p \colon \N_f(\D) \to \N(\D)$ down to the ordinary nerve of $\D$, induced by the unique map to the terminal object $\Delta^0 \in \sSet$ \cite{htt}*{3.2.5.4}.
	When $f$ takes values in quasicategories, this canonical map is a coCartesian fibration \emph{classified} (Definition \ref{def:classified}) by $\N(f)$:

	\begin{proposition}[\cite{htt}*{3.2.5.21}]
		\label{prop:rel-nerve-infty-gr}
		Let $f \colon \D \to \sSet$ be a functor such that each $fd$ is a quasicategory.
		Then:
		\begin{enumerate}[label = (\roman*)]
			\item $p \colon \N_f(\D) \to \N(\D)$ is a coCartesian fibration of simplicial sets, and
			\item $p$ is classified by the functor $\N(f) \colon \N(\D) \to \Cat_\infty$, i.e.~there is an equivalence of coCartesian fibrations
			 \[
			 	\N_f(\D) \simeq \Gr_\infty \N(f).
			 \]
	 	\end{enumerate}
	\end{proposition}

\begin{remark}
	Note that the version of Proposition \ref{prop:rel-nerve-infty-gr} in \cite{htt}  is somewhat ambiguously stated. In particular, it is claimed that, given a functor $f\colon\D\to \sSet$, the fibration $\N_{f}(\D)$ is the one \textit{associated} to the functor $\N(f)\colon N(\D)\to \Cat_\infty$. However, a close reading of the proof given in \cite{htt} makes it clear that, for a functor $f\colon\D\to \sSet$ with associated $f^\natural\colon\D\to \sSet^+$, there is an equivalence $\N_f(\D)^\natural\simeq \N_{f^\natural}^+(\D)\simeq \Un_\phi^+f^\natural$.
	Here, $\N^+_{f^\natural}$ indicates the \textit{marked} analog of the relative nerve described in Definition \ref{def:relnerve}. Application of the (large) simplicial nerve functor recovers the form of the proposition given above.
\end{remark}

	\subsection{The Grothendieck construction $\Gr F$}

	Suppose instead that we have a functor $F \colon \D \to \sCat$. 
	We may then take the nerve relative to the composite $f \colon \D \xrightarrow{F} \sCat \xrightarrow{\N} \sSet$ to get a coCartesian fibration $\N_{f}(\D) \to \N(\D)$.
	We now describe a second way to obtain a coCartesian fibration over $\N(\D)$ from such an $F$.

	\begin{definition}[\cite{beardswong}*{Definition 4.4}]
		Let $\D$ be a small category, and let $F \colon \D \to \sCat$ be a functor.
		The \fm{Grothendieck construction of $F$} is the simplicial category $\Gr F$ with objects and morphisms:
		\begin{align*}
			\Ob(\Gr F) &:= \coprod_{\;\;\, d \in \D \;\;\,} \Ob(Fd) \times \{d\}, \\
			\Gr F\big( (x,c), (y,d) \big) &:= \coprod_{\varphi \colon c \to d} Fd(F\varphi\; x, y) \times \{\varphi\} .
		\end{align*}
	An arrow $(x,c) \to (y,d)$ (i.e.~a $0$-simplex in $\Gr F( (x,c), (y,d))$) is a pair $\left( F\varphi\;x \xrightarrow{\sigma} y, c \xrightarrow{\varphi} d \right)$, while
	the composite $(x,c) \xrightarrow{(\sigma, \varphi)} (y,d) \xrightarrow{(\tau, \psi)} (z,e)$ is 
	\[
		\bigg( F(\psi \varphi)\, x = 
		F\psi\, F\varphi\, x\xto{F\psi\, \sigma} F\psi\, y \xto{\tau} z
			\;,\;\;
		c \xto{\varphi} d \xto{\psi} e
		\bigg).
	\]		
	\end{definition}

	There is a simplicial functor $P \colon \Gr F \to \D,\; (x,c) \mapsto c,$
	induced by the unique maps $Fd(F\varphi\; x, y) \to \Delta^0$. 
	Here, $\D$ is treated as a \emph{discrete} simplicial category with hom-objects
	\[
		\D(c,d) = \coprod_{\varphi \colon c \to d} \Delta^0 \times \{\varphi\}.
	\]

	\begin{definition}[\cite{beardswong}*{Definition 3.5, Proposition 3.6}]
		Let $P \colon \E \to \D$ be a simplicial functor.
		A map $\chi \colon e \to e'$ in $\E$ is \fm{$P$-coCartesian} if 
		\begin{equation} \label{eq:opfib-pullback}
			\begin{tikzcd}[column sep = large]
				\E(e',x) \ar[r, "-\circ \chi"] \ar[d, "P_{e'x}"'] & \E(e,x) \ar[d, "P_{ex}"]
				\\
				\D(Pe', Px) \ar[r, "-\circ P\chi"] & \D(Pe, Px)
			\end{tikzcd}
		\end{equation}		
		is a (ordinary) pullback in $\sSet$ for every $x \in \E$.

		A simplicial functor $P \colon \E \to \D$ is a \fm{simplicial opfibration} if for every $e \in \E, d \in \D$ and $\varphi \colon Pe \to d$, there exists a $P$-coCartesian lift of $\varphi$ with domain $e$.
	\end{definition}

	\begin{proposition}[\cite{beardswong}*{Proposition 4.11}]\label{prop:bw411}
		The functor $\Gr F \to \D$ is a simplicial opfibration.
	\end{proposition}

	\begin{proposition}\label{prop:opfibtococart}
		Let $\D$ be a category (i.e.\ a discrete simplicial category), and $\E$ be a locally Kan simplicial category. 
		If $P \colon \E \to \D$ is a simplicial opfibration, then $\N(P) \colon \N(\E) \to \N(\D)$ is a coCartesian fibration.
	\end{proposition}
	\begin{proof}
		It suffices to show that any $P$-coCartesian arrow in $\E$ gives rise to a $\N(P)$-coCartesian arrow in $\N(\E)$.
		If $\chi \colon e \to e'$ is $P$-coCartesian, then (\ref{eq:opfib-pullback})
		is an ordinary pullback in $\sSet$ for all $x \in \E$.
		Since $\D(Pe, Px)$ is discrete and $\E(e,x)$ is fibrant, $P_{ex}$ is a fibration\footnote{Any map into a coproduct of simplicial sets induces a coproduct decomposition on its domain (by taking fibers over each component of the codomain). Since all horns $\Lambda^n_k$ are connected, any commuting square from a horn inclusion to $P_{ex}$ necessarily factors through one of the components of $\E(e,x)$, and may thus be lifted because $\E(e,x)$ is fibrant.};
		since $\D(Pe', Px)$ is also fibrant, this ordinary pullback is in fact a \emph{homotopy} pullback \cite{htt}*{A.2.4.4}.
		Thus, by \cite{htt}*{2.4.1.10}, $\chi$ gives rise to a $\N(P)$-coCartesian arrow in $\N(\E)$.
	\end{proof}

	\begin{remark}
		The discreteness of $\D$ and fibrancy of $\E$ are critical here.
		An arbitrary $\sSet$-enriched opfibration $P \colon \E \to \D$ is unlikely to give rise to a coCartesian fibration $\N(P) \colon \N(\E) \to \N(\D)$.
		Essentially, we require the ordinary pullback in (\ref{eq:opfib-pullback}) to be a homotopy pullback.
	\end{remark}	

	\begin{corollary}
		Let $\D$ be a small category and $F \colon \D \to \sCat$ be such that each $Fd$ is locally Kan.
		Then $\N(\Gr F) \to \N(\D)$ is a coCartesian fibration.
	\end{corollary}

	\subsection{Comparing $\N(\Gr F)$ and $\N_f(\D)$}
	\begin{theorem}
		\label{thm:gr-rel-nerve}
		Let $F \colon \D \to \sCat$ be a functor, and $f = \N F$.
		Then there is an isomorphism of coCartesian fibrations
		\[
			\N(\Gr F) \cong \N_f(\D).
		\]
	\end{theorem}
	\begin{proof}
		We will only explicitly describe the $n$-simplices of $\N(\Gr F)$ and $\N_f(\D)$ and show that they are isomorphic.
		From the description, it should be clear that we do indeed have an isomorphism of simplicial sets that is compatible with their projections down to $\N(\D)$, hence an isomorphism of coCartesian fibrations (by \cite{riehl2017fibrations}*{5.1.7}, for example).

		\paragraph{Description of $\N(\Gr F)_n$.}
		An $n$-simplex of $\N(\Gr F)$ is a simplicial functor $S \colon \CC[\Delta^n] \to \Gr F$.
		By Lemma \ref{lem:simplicial-functor}, this is the data of:
		\begin{itemize}
			\item for each $i \in [n]$, an object $S_i = (x_i, d_i) \in \Gr F$, (so $d_i \in \D, x_i \in Fd_i$)
			\item for each $r$-dimensional bead shape $\bead{I_0 | \dots | I_r}$ of  $\{i_0 < \dots < i_m\} \subseteq [n]$ where $m\geq 1$, an $r$-simplex
			\[
				S_\bead{I_0 | \dots | I_r} \in \Gr F (S_{i_0}, S_{i_m}) = \coprod_{\varphi \in \D(d_{i_0}, d_{i_m})} Fd_{i_m}(F\varphi\; x_{i_0}, x_{i_m})
			\]
			whose boundary is compatible with lower-dimensional data.
		\end{itemize}		

		\paragraph{Description of $\N_f(\D)_n$.}
		An $n$-simplex of $\N_f(\D)$ consists of a functor $d\colon [n] \to \D$, picking out objects and arrows $d_i \xrightarrow{d_{ij}} d_j$ for all $0 \leq i \leq j \leq n$ such that $d_{ii}$ are identities and
		\[
			d_{jk}d_{ij} = d_{ik}, \quad i \leq j \leq k,
		\]
		and a family of maps $s^J \colon \Delta^J \to fd_j$ for every $J \subseteq [n]$ with maximal element $j$, satisfying (\ref{eq:rel-nerve}).
		Since $f = \N F$, such maps $s^J \colon \Delta^J \to \N Fd_j$ correspond, under the $\CC \dashv \N$ adjunction, to maps $S^J \colon \CC[\Delta^J] \to Fd_j$ satisfying:
		\begin{equation} \label{eq:rel-nerve-transpose}
			\begin{tikzcd}
				\CC[\Delta^I] \ar[r, "S^I"]  \ar[d, hookrightarrow] & Fd_i \ar[d, "F d_{ij}"]
				\\
				\CC[\Delta^J] \ar[r, "S^J"] & Fd_j
			\end{tikzcd}
		\end{equation}
		By Lemma \ref{lem:simplicial-functor}, each $S^J$ is the data of:
		\begin{itemize}
			\item for each $i \in J$, an object $S^J_i \in Fd_j$
			\item for each $r$-dimensional bead shape $\bead{I_0 | \dots | I_r}$ of $\{i_0 < \dots < i_m\} \subseteq J$ where $m \geq 1$, an $r$-simplex
			\[
				S^J_\bead{I_0|\dots | I_r} \in Fd_j(S^J_{i_0}, S^J_{i_m})
			\]
			whose boundary is compatible with lower-dimensional data.
		\end{itemize}
		The condition (\ref{eq:rel-nerve-transpose}) is equivalent to
		\begin{align} \label{eq:rel-nerve-explicit}
			Fd_{ij}\, S^I_k &= S^J_k, &\text{and} & & Fd_{ij}\, S^I_\bead{I_0|\dots|I_r} &= S^J_\bead{I_0|\dots | I_r}.
		\end{align}
		for any $k \in I$ and bead shape $\bead{I_0|\dots | I_r}$ of $I \subseteq J$.

		\paragraph{From $\N(\Gr F)_n$ to $\N_f(\D)_n$.}
		Given $S \colon \CC[\Delta^n] \to \Gr F$, we first produce a functor $d \colon [n] \to \D$.
		For any $\{i < j\} \subseteq [n]$, we have a $0$-simplex 
		\[
			S_\bead{ij} = (Fd_{ij} x_i \xrightarrow{x_{ij}} x_j , d_i \xrightarrow{d_{ij}} d_j) \in \Gr F \big((x_i, d_i), (x_j, d_j)\big)_0,
		\]
		and for any $\{i < j < k\} \subseteq [n]$, we have a $1$-simplex $S_\bead{ik|j}$ from $S_\bead{ik}$ to
		\[
			S_\bead{jk}S_\bead{ij} = (Fd_{jk} Fd_{ij} x_i  \xrightarrow{Fd_{jk} x_{ij}} Fd_{jk} x_j \xrightarrow{x_{jk}} x_k\;,\;\; d_i \xrightarrow{d_{ij}} d_j \xrightarrow{d_{jk}} d_k ).
		\]
		But such a $1$-simplex includes the data of a $1$-simplex from $d_{ik}$ to $d_{jk}d_{ij}$ in the \emph{discrete} simplicial set $\D(x_i, x_k)$.
		Thus $d_{ik}$ must be \emph{equal} to $d_{jk} d_{ij}$, so the data of $\{ d_i \xrightarrow{d_{ij}} d_j \}_{i \leq j}$, where $d_{ii}$ is the identity, assembles into a functor $d \colon [n] \to \D$ as desired.
		Note that since $F$ is a functor, we also have
		\[
			Fd_{jk}\; Fd_{ij} = F(d_{jk}d_{ij}) = Fd_{ik}.
		\]		
		Next, for each non-empty subset $J \subseteq [n]$ with maximal element $j$, we need a simplicial functor $S^J \colon \CC[\Delta^J] \to Fd_j$.
		For each $i \in J$, set
		\[
			S^J_i := Fd_{ij}\, x_i \in Fd_j.
		\]
		For each $r$-dimensional bead shape $\bead{I_0|\dots|I_r}$ of $\{i_0 < \dots < i_m\} \subseteq J$ with $m \geq 1$, we first note that $S_\bead{I_0|\dots |I_r}$ lies in the $d_{i_0 i_m}$ component
		\[
			Fd_{i_m}(Fd_{i_0 i_m}\, x_{i_0}, x_{i_m}) \subset \Gr F(S_{i_0}, S_{i_m})
		\]
		because its sub-simplices (for instance $S_\bead{i_0 i_m}$) do too. 
		Define
		\[
			S^J_\bead{I_0 | \dots | I_r} := Fd_{i_m j}\, S_\bead{I_0|\dots |I_r}.
		\]
		We verify that this lives in the correct simplicial set
		\begin{align*}
			Fd_j(Fd_{i_m j}\; Fd_{i_0 i_m}\; x_{i_0}, Fd_{i_m j}\; x_{i_m}) &= Fd_j(Fd_{i_0 j} x_{i_0}, Fd_{i_m j} x_{i_m}) \\
			&= Fd_j(S^J_{i_0}, S^J_{i_m}).
		\end{align*}
		The boundary of each $S^J_\bead{I_0 | \dots | I_r}$ is compatible with lower-dimensional data because the boundary of each $S_\bead{I_0 | \dots | I_r}$ is as well.
		We thus get a simplicial functor $S^J \colon \CC[\Delta^J] \to Fd_j$, and by construction, the functoriality of $F$ and $d$ implies that (\ref{eq:rel-nerve-explicit}) holds.

		\paragraph{From $\N_f(\D)_n$ to $\N(\Gr F)_n$.}
		Conversely, suppose we have $d \colon [n] \to \D$ and $S^J \colon \CC[\Delta^J] \to Fd_j$ for every non-empty $J \subseteq [n]$ with maximal element $j$, satisfying (\ref{eq:rel-nerve-explicit}).		
		For each $i \in [n]$, let $S_i := (S^{\{i\}}_i, d_i)$, 
		and for each $r$-dimensional bead shape $\bead{I_0 | \dots | I_r}$ of $I = \{i_0,\dots,i_m\} \subseteq [n]$ where $m \geq 1$, let
		\[
			S_\bead{I_0| \dots |I_r} := S^I_\bead{I_0| \dots I_r}.
		\]
		Then $S_\bead{I_0| \dots | I_r}$ is an $r$-simplex in 
		\[
 			 Fd_{i_m}(S^I_{i_0}, S^I_{i_m}) = Fd_{i_m}(Fd_{i_0 i_m}\; S^{\{i_0\}}_{i_0}, S^{\{i_m\}}_{i_m})  \subset \Gr F(S_{i_0}, S_{i_m})
		\]
		as desired,	where we have used (\ref{eq:rel-nerve-explicit}) in the first equality,
		and this data yields a simplicial functor $S \colon \CC[\Delta^n] \to \Gr F$.

		\paragraph{Mutual inverses.}
		Finally, it is easy to see that the constructions described above are mutual inverses.
		For instance, we have
		\begin{align*}
			S_\bead{I_0 | \dots | I_r} &= Fd_{ii}\, S_\bead{I_0 | \dots | I_r}, \\
			S^J_\bead{I_0 | \dots | I_r} &= Fd_{ij}\, S^I_\bead{I_0 | \dots | I_r}.
		\end{align*}
		Thus $\N(\Gr F)_n \cong \N_f(\D)_n$. 
	\end{proof}

	In light of Proposition \ref{prop:rel-nerve-infty-gr}, we obtain:
	\begin{corollary}
	\label{cor:gr-infty-gr}
		Let $F \colon \D \to \sCat$ be a functor such that each $Fd$ is a quasicategory, and $f = \N F$. 
		Then there is an equivalence of coCartesian fibrations
		\[
			\N(\Gr F) \simeq \Gr_\infty \N(f).
		\]
	\end{corollary}

\section{Operadic nerves of monoidal simplicial categories}
\label{sec:monoid-struct}
	Given a monoidal simplicial category $\C$, \cite{dag2}*{1.6} describes the formation of a simplicial category $\C^\otimes$ equipped with an opfibration over $\Delta^\op$.
	The nerve of this opfibration is a coCartesian fibration $\N(\C^\otimes) \to \N(\Delta^\op)$ which has the structure of a monoidal quasicategory in the sense of \cite{dag2}*{1.1.2}.
	Since this construction is exactly the operadic nerve of \cite{ha}*{2.1.1} applied to the underlying simplicial operad of $\C$, we call $\N^\otimes(\C) := \N(\C^\otimes)$ the \emph{operadic nerve of a monoidal simplicial category $\C$}. 

	In this section, we apply the results of the previous section to further describe the process of obtaining $\N^\otimes(\C)$ from a \emph{strict} monoidal $\C$.
	We show that the opfibration $\C^\otimes \to \Delta^\op$ is the Grothendieck construction $\Gr\, \C^\bullet$ of a functor $\C^\bullet \colon \Delta^\op \to \sCat$, and hence conclude that the operadic nerve $\N^\otimes(\C)$ is the nerve of $\Delta^\op$ relative to $\Delta^\op \xrightarrow{\C^\bullet} \sCat \xrightarrow{\N} \sSet.$

	Although the operadic nerve may be defined for any monoidal simplicial category $\C$, we restrict the discussion in this section to \emph{strict} monoidal categories because the results of the previous section require strict functors $\D \to \sCat$ and $\D \to \sSet$ rather than pseudofunctors.

\subsection{$\C^\otimes$ and $\C^\bullet$ from a strict monoidal $\C$}
	We start by describing the opfibration $\C^\otimes \to \Delta^\op$ and the functor $\C^\bullet \colon \Delta^\op \to \sCat$ associated to a strict monoidal simplicial category $\C$.

	\begin{definition}
		A \fm{strict monoidal simplicial category} $\C$ is a monoid in $(\sCat, \times, *)$.
		Let $\otimes \colon \C \times \C \to \C$ denote the monoidal product of $\C$ and $\1 \colon * \to \C$ denote the monoidal unit, which we identify with an object $\1 \in \C$. Let $\sMonCat$ denote the category of strict monoidal simplicial categories, which is equivalently the category of monoids in $\sCat$.
	\end{definition}

	A strict monoidal simplicial category is thus a simplicial category with a strict monoidal structure that is \emph{weakly compatible} in the sense of \cite{dag2}*{1.6.1}.
	The strictness of the monoidal structure implies that we have equalities (rather than isomorphisms):
	\begin{align*}
		(x\otimes y) \otimes z &= x \otimes (y \otimes z), & 	\1 \otimes x &= x = x \otimes \1.
	\end{align*}

	\begin{definition}[\cite{dag2}*{1.1.1}]
		Let $(\C,\otimes, \1)$ be a strict monoidal simplicial category. 
		Then we define a new category $\C^\otimes$ as follows:
		\begin{enumerate}
			\item An object of $\C^\otimes$ is a finite, possibly empty, sequence of objects of $\C$, denoted $[x_1,\ldots,x_n].$
			\item The simplicial set of morphisms from $[x_1,\ldots,x_n]$ to $[y_1,\ldots,y_m]$ in $\C^\otimes$ is defined to be
			\[
			\coprod_{f \in \Delta\left([m],[n]\right)}\; \prod_{1\leq i\leq m} \C\big(x_{f(i-1)+1}\otimes x_{f(i-1)+2} \otimes \cdots\otimes x_{f(i)}\;,\;\; y_i\big)
			\] 
			 where $x_{f(i-1)+1}\otimes \cdots\otimes x_{f(i)}$ is taken to be $\1$ if $f(i-1) = f(i)$.

			 A morphism will be denoted $[f; f_1,\dots, f_m]$, where
			 \[
			 	x_{f(i-1)+1}\otimes \cdots\otimes x_{f(i)} \xto{\quad f_i \quad} y_i.
			 \]
			\item Composition in $\C^{\otimes}$ is determined by composition in $\Delta$ and $\C$:
			\begin{align*}
				[g; g_1, \dots g_\ell] \circ [f; f_1,\dots, f_m] &= [f\circ g\;; \;\; h_1, \dots, h_\ell], \\
				\text{where} \quad h_i &= g_i \circ (f_{g(i-1)+1} \otimes \dots \otimes f_{g(i)}).
			\end{align*}
			This is associative and unital due to the associativity and unit constraints of $\otimes$. 
		\end{enumerate}
	\end{definition}

	\begin{remark}
		Though we don't make it explicit here, $\C^\otimes$ is the category of operators (in the sense of \cite{maythom} and \cite{gephaugenriched}*{2.2.1}) of the underlying simplicial multicategory (cf.~\cite{gephaugenriched}*{3.1.6}) of $\C$. 
	\end{remark}

	There is a forgetful functor $P \colon \C^\otimes \to \Delta^\op$ sending $[x_1,\dots, x_n]$ to $[n]$
	which is an (unenriched) opfibration of categories \cite{dag2}*{1.1(M1)}.
	The proof of that statement can easily be modified to show:

	\begin{proposition}\label{prop:Cotimesopfib}
	The functor $P \colon \C^\otimes\to \Delta^\op$ is a simplicial opfibration.
	\end{proposition}
	\begin{proof}
		Replace all hom-sets by hom-\emph{simplicial}-sets in \cite{dag2}*{1.1(M1)}.
	\end{proof}

	In fact, we may choose $P$-coCartesian lifts so that $P$ is a \emph{split} simplicial opfibration\footnote{This essentially means that $\C^\bullet$ is a functor rather than a pseudofunctor. Note that if $\C$ is not strictly monoidal, then $x_{f(i-1)+1} \otimes \dots \otimes x_{f(i)}$ is not well-defined: a choice of parentheses needs to be made. Although the various choices are isomorphic, they are not identical, and this obstructs our ability to obtain a split opfibration.}:
	given $[x_1, \dots, x_n] \in \C^\otimes$ and a map $f \colon [m] \to [n]$, let 
	\begin{equation} \label{eq:yi}
		y_i = x_{f(i-1)+1} \otimes \dots \otimes x_{f(i)}
	\end{equation}
	for all $1 \leq i \leq m$. Then $[f; 1_{y_1}, \dots, 1_{y_m}]$ is a $P$-coCartesian lift of $f$.

	By the enriched Grothendieck correspondence \cite{beardswong}*{Theorem 5.6}, the split simplicial opfibration $P\colon \C^\otimes \to \Delta^\op$ with this choice of coCartesian lifts arises from a functor $\C^\bullet \colon \Delta^\op \to \sCat$ which we now describe.

	\begin{definition} \label{def:C-faces}
		Let $\C$ be a strict monoidal simplicial category with monoidal product, unit and terminal morphisms the simplicial functors $\mu_{\C}\colon \C\times\C\to C$, $\eta\colon \1\to \C$ and $\varepsilon\colon \C\to \1$ respectively. Then for each $0\leq i\leq n$ we define the functor $\C^{\delta_i}\colon \C^n\to \C^{n-1}$ to be:
		\begin{enumerate}[label=(\roman*)]
			\item the application of $\mu_{\C}$ to the $i^{th}$ and $i+1^{st}$ coordinates of $\C^n$, and the identity in all other coordinates, in the case that $0<i<n$;
			\item the application of $\varepsilon$ to the $i^{th}$ coordinate and the identity in all other coordinates in the case that $i\in\{0,n\}$.
		\end{enumerate}
	In the other direction, for each $0\leq i\leq n$, we define a functor $\C^{\sigma_i}\colon C^{n-1}\to\C^{n}$ to be the isomorphism $\C^n\cong\C^{i}\times \1\times \C^{n-i}$ followed by the application of the unit $\eta$ in the $i^{th}$ coordinate $\C^{i}\times\1\times\C^{n-i}\to \C^{n}$. 
	\end{definition}

	\begin{definition} \label{def:C-bullet}
		Let $\C$ be a strict monoidal simplicial category. Then define the functor $\C^{\bullet}\colon \Delta^{op}\to \sCat$ to be the one that takes $[n]$ to $\C^n$, the face maps $\delta_i\colon [n-1]\to[n]$ to the functors $\C^{\delta_i}:\C^n\to \C^{n-1}$ and the degeneracy maps $\sigma_i\colon [n]\to [n-1]$ to the functors $\C^{\sigma_i}\colon \C^{n-1}\to \C^{n}$, where $\C^{\delta_i}$ and $\C^{\sigma_i}$ are as in Definition \ref{def:C-faces}. 
	\end{definition}

	\begin{remark}
		The fact that $\C$ is a strict monoid in $\sCat$ implies that the functors $\mu_i$ and $\eta_i$ satisfy the simplicial identities. This is not difficult to check but is tedious, so we will not include a proof of it. 
	\end{remark}

		\begin{remark}\label{rem:Cf}
			More generally, let $f \colon [m] \to [n]$ be a morphism in $\Delta$. Then by decomposing $f$ into a finite composition of face and degeneracy maps, we have that $\C^f \colon \C^n \to \C^m$ is the functor that sends $(x_1,\dots, x_n)$ to $(y_1, \dots, y_m)$ where $y_i$ is given by (\ref{eq:yi}), and (when restricted to zero simplices) sends $(\varphi_1, \dots, \varphi_n)$ to $(\psi_1, \dots, \psi_m)$ where \[ \psi_i = \varphi_{f(i-1)+1} \otimes \dots \otimes \varphi_{f(i)}.\]
		\end{remark}

	\begin{lemma} \label{lem:C-otimes-bullet}
		For a strict monoidal simplicial category $\C$, there is an isomorphism of simplicial categories
		\[
			\C^\otimes \cong \Gr\, \C^\bullet.
		\]
	\end{lemma}
	\begin{proof}
		This follows directly from the definitions of $\C^\otimes$, $\C^\bullet$ and $\Gr$. Explicitly, first notice that there is a bijection on objects $F\colon \Ob(C^\otimes)\to \Ob (\Gr \C^\bullet)$ given by \[
		F([x_1,\ldots,x_n])=((x_1,\ldots,x_n),[n])\in\!\!\! \coprod_{\;\;\, [m] \in \Delta^{op} \;\;\,} \Ob(\C^n)\times\{[m]\}.
		\]
		The space of morphisms from $((x_1,\ldots,x_m),[m])$ to $((y_1,\ldots,y_n),[n])$ in $\Gr(\C^\bullet)$ is, by definition, the coproduct 
		\[
		\coprod_{\varphi \colon [n] \to [m]} \C^n(\C^\varphi(x_1,\ldots,x_m),(y_1,\ldots, y_n))\times\{\varphi\},
		\] which is clearly isomorphic to 
		\[
		\coprod_{\varphi \colon [n] \to [m]} \C^n(\C^\varphi(x_1,\ldots,x_m),(y_1,\ldots, y_n)).
		\]
		By using Definition \ref{def:C-bullet}, Remark \ref{rem:Cf} and the fact that the mapping spaces of a product of categories are the product of mapping spaces, it is easy to see that this last expression is equal to 
		\[
		\coprod_{\varphi\colon [n]\to[m]}\; \prod_{1\leq i\leq m} \C\big(x_{\varphi(i-1)+1}\otimes x_{\varphi(i-1)+2} \otimes \cdots\otimes x_{\varphi(i)}\;,\;\; y_i\big)
		\] 
	\end{proof}

	\begin{remark}
		In fact, the results of this subsection hold more generally for monoidal $\V$-categories, where $\V$ satisfies the hypotheses of \cite{beardswong}, but we will not need this level of generality.
	\end{remark}

\subsection{The operadic nerve $\N^\otimes$}
	We now suppose that $\C$ is a strict monoidal \emph{fibrant} (i.e.\ locally Kan) simplicial category.
	Then $\C^\otimes$ is a fibrant simplicial category as well, so the simplicial nerves of  $\C$ and $\C^\otimes$ are both quasicategories.

	\begin{definition}
		Let $(\C, \otimes)$ be a strict monoidal fibrant simplicial category. 
		The \fm{operadic nerve of $\C$ with respect to $\otimes$} is the quasicategory
		\[
			\N^\otimes(\C) := \N(\C^\otimes).
		\]
	\end{definition}
	Combining Propositions \ref{prop:opfibtococart} and \ref{prop:Cotimesopfib} with $p := \N(P)$, we obtain:
	\begin{corollary}
		There is a coCartesian fibration $p \colon \N^\otimes(\C) \to \N(\Delta^\op)$.
	\end{corollary}

	In fact, $p$ defines a monoidal structure on $\N(\C)$ in the following sense:

	\begin{definition}[\cite{dag2}*{1.1.2}]	
		A \fm{monoidal quasicategory} is a coCartesian fibration of simplicial sets $p:X\to N(\Delta^\op)$ such that for each $n \geq 0$, the functors $X_{[n]}\to X_{\{i,i+1\}}$ induced by $\{i, i+1\} \hookrightarrow [n]$ determine an equivalence of quasicategories
		\[
			X_{[n]}\xto{\quad \simeq \quad} X_{\{0,1\}}\times\cdots\times X_{\{n-1,n\}} \cong (X_{[1]})^n,
		\]
		where $X_{[n]}$ denotes the fiber of $p$ over $[n]$.
		In this case, we say that $p$ defines a \fm{monoidal structure on $X_{[1]}$}.
	\end{definition}	

	\begin{proposition}[\cite{dag2}*{Proposition 1.6.3}]
		If $\C$ is a strict monoidal fibrant simplicial category then $p \colon \N^\otimes(\C) \to \N(\Delta^\op)$ defines a monoidal structure on the quasicategory $\N(\C)\cong (\N^\otimes(\C))_{[1]}$. 
	\end{proposition}

	\begin{definition}
		The \fm{quasicategory of monoidal quasicategories} is the full subquasicategory $\MonCat_\infty \subset \coCart_{/\N(\Delta^\op)}$ containing the monoidal quasicategories.
	\end{definition}

	\begin{definition}
		Let $\C$ be a strict monoidal fibrant simplicial category.
		The \fm{vertex associated to $\C$} in $\MonCat_\infty$ or $\coCart_{/\N(\Delta^\op)}$ is the vertex corresponding to $p \colon \N^\otimes(\C) \to \N(\Delta^\op)$.
	\end{definition}

	\begin{remark}
		By Definition \ref{def:cocartqcat}, the vertex associated to $\C$ is equivalently the vertex corresponding to $\N^\otimes(\C)^\natural \to \N(\Delta^\op)^\sharp$ in $\N \big((\sSet^+)_{/S}\big)^\circ$.
		Note that, by \cite{htt}*{3.1.4.1}, the assignment $(X \to S) \mapsto (X^\natural \to S^\sharp)$ is injective up to isomorphism. 
		
		
	\end{remark}

	Finally, we tie together the results of this and the previous sections.

	\begin{corollary} \label{cor:NC-GrC}
		Let $\C$ be a strict monoidal fibrant simplicial category,
		and let $\xi$ be the composite $\Delta^\op \xto{\C^\bullet} \sCat \xto{\N} \sSet$. 
		Then we have the following string of isomorphisms and equivalences:
	\begin{equation} \label{eq:op-rel-nerve}
		\N^\otimes(\C) \cong \N(\Gr\, \C^\bullet) \cong \N_{\xi}(\Delta^\op) \simeq \Gr_\infty\N(\xi).
	\end{equation}
	\end{corollary}

	\begin{remark}
		The preceding Corollary and the $\infty$-categorical Grothendieck correspondence (\ref{cor:gr-corr-infty}) suggest that we may equivalently define a monoidal quasicategory to be $\xi \in (\Cat_\infty)^{\N(\Delta^\op)}$ such that the maps 
		\[
			\xi([n]) \xto{ \xi\left(\{i, i+1\} \hookrightarrow [n]\right) } \xi(\{i,i+1\})
		\]
		induce an equivalence
		\[
			\xi({[n]}) \xto{\quad \simeq \quad} \xi({\{0,1\}}) \times\cdots\times \xi({\{n-1,n\}}) \cong \xi({[1]})^n.
		\]
	\end{remark}

	\begin{remark}
		We have worked entirely on the level of \emph{objects} as we are only interested in understanding the operadic nerve of one monoidal simplicial category at a time.
		However, we believe it should be possible to show that these constructions and equivalences are \emph{functorial}, so that the following diagram is an actual commuting diagram of functors between appropriately defined categories or quasicategories:
		\[
		\begin{tikzcd}
			\sMonCat \ar[r, "(-)^\bullet"] \ar[rr, bend right = 15, "(-)^\otimes"' description, near end] \ar[rrr, bend right = 20, "\N^\otimes"' description, near end]  & \sCat^{\Delta^\op} \ar[r, "\Gr"] & \opFib_{/\Delta^\op} \ar[r, "\N"] & \coCart_{/\N(\Delta^\op)}
		\end{tikzcd}
		\]

		For an ordinary category $\D$, we also believe that there is a model structure on $\sCat_{/\D}$ whose fibrant objects are simplicial opfibrations (or the analog for a suitable version of \textit{marked} simplicial categories), along with a Quillen adjunction between $\sCat_{/D}$ and $(\sSet^+)_{/\N(\D)}$ whose restriction to fibrant objects picks out the maps arising as nerves of simplicial opfibrations. 

		
	\end{remark}

\section{Opposite functors}
\label{sec:op-func}

Finally, we turn to the question which motivated this paper:~how does the operadic nerve interact with taking opposites?

Recall that there is an involution on the category of small categories $\op\colon\Cat\to\Cat$ which takes a category to its opposite. There are higher categorical generalizations of this functor to the category of simplicial sets and the category of simplicially enriched categories, which we review in turn.

\subsection{Opposites of (monoidal) simplicial categories}

\begin{definition}
	Given a simplicial category $\C \in \sCat$, let $\C^\op$ denote the category with the same objects as $\C$, and morphisms
	\[
		\C^\op(x,y) := \C(y,x).
	\]
	Let $\op_s\colon \sCat\to \sCat$ be the functor sending $\C$ to $\op_s(\C) := \C^\op$, and sending a simplicial functor $F$ to the simplicial functor $F^\op$ given by $F^\op x:= Fx$ and $F^\op_{x,y} := F_{y,x}$.
\end{definition}

We note a few immediate properties of opposites.

\begin{lemma} \label{lem:op-s-selfadj}
	The functor $\op_s$ is self-adjoint.
\end{lemma}

\begin{lemma}
	Let $\C$ be a simplicial category.
	If $\C$ is fibrant, then so is $\C^\op$.
\end{lemma}

\begin{lemma}
	Let $\C$ be a strict monoidal simplicial category. 
	Then $\C^\op$ is canonically a strict monoidal simplicial category as well. 
\end{lemma}
\begin{proof}	
	Given $x,y \in \C^\op$, define their tensor product to be the same object as their tensor in $\C$.
	One can check that this extends to a monoidal structure on $\C^\op$.

	Alternatively, since $\op_s$ is self-adjoint, it preserves limits and colimits of simplicial categories. 
	In particular, it preserves the Cartesian product, and is therefore a monoidal functor from $(\sCat, \times)$ to itself.
	It thus preserves monoids in $\sCat$.
\end{proof}

\begin{remark}
	Since the same object represents the tensor product of $x$ and $y$ in $\C$ or $\C^\op$, we will use the same symbol $\otimes$ to denote the tensor product in either category.
\end{remark}

The functor $\op_s \colon \sCat \to \sCat$ induces functors $(-)^\op \colon \sMonCat \to \sMonCat$ and $(-)^\op \colon \sCat^{\Delta^\op} \to \sCat^{\Delta^\op}$, where the latter is composition with $\op_s$. 
We wish to show that these functors commute with the construction $\C \mapsto \C^\bullet$ of Definition \ref{def:C-bullet}.
\begin{lemma} \label{lem:C-bullet-op}
	Let $\C$ be a strict monoidal simplicial category.
	Then 
	\[
		(\C^\bullet)^\op = (\C^\op)^\bullet,
	\]
	i.e.\ the following diagram commutes on objects.
	\[
		\begin{tikzcd}[row sep = large]
			\sMonCat \ar[r, "(-)^\bullet"] \ar[d, "\op" description] & \sCat^{\Delta^\op} \ar[d, "\op" description] 
			\\
			\sMonCat \ar[r, "(-)^\bullet"] & \sCat^{\Delta^\op} 		
		\end{tikzcd}
\]
\end{lemma}
\begin{proof}
	The objects of both $(\C^n)^\op$ and $(\C^\op)^n$ are $n$-tuples $(x_1, \dots, x_n)$ where $x_i \in \C$, while the simplicial set of morphisms from $(x_1, \dots, x_n)$ to $(y_1,\dots, y_n)$ are both $\C(y_1,x_1) \times \dots \times \C(y_n, x_n)$, so $(\C^n)^\op = (\C^\op)^n$. Therefore $(\C^{op})^\bullet$ and $(\C^\bullet)^{op}$ agree on objects $[n]\in\Delta^{op}$. 
	
	Now consider the face and degeneracy morphisms of $\Delta$ under the two functors $(\C^\bullet)^{op}\colon \Delta^{op}\to\sCat$ and $(\C^{op})^\bullet\colon \Delta^{op}\to\sCat$. In the first case, they are taken to, respectively, an application of the opposite monoidal structure to the $i^{th}$ and $i+1^{st}$ coordinates $(\C^{\delta_i})^{op}\colon (\C^n)^{op}\to(\C^{n-1})^{op}$ and an application of the ``opposite'' unit in the $i^{th}$ coordinate  $(\C^{\sigma_i})^{op}\colon (\C^{n-1})^{op}\to (\C^{n})^{op}$. Because the monoidal structure of $\C^{op}$ is by definition the opposite of the monoidal structure of $\C$, and both the identity maps and the unit maps are self-dual under $op$ (and the fact that $op$ is self-adjoint so preserves products up to equality), it is clear that these are equal to $(\C^{op})^{\delta_i}$ and $(\C^{op})^{\sigma_i}$ respectively.
	
\end{proof}

\begin{remark}
	The diagram above is an actual commuting diagram of functors, but we will not show this here, since we have not fully described the functorial nature of $(-)^\op$.
\end{remark}

We note also that opposites commute with the simplicially enriched Grothendieck construction, but we will not need this result in the rest of the paper.

\begin{definition} \label{def:sfib-op}
	Let $P \colon \E \to \D$ be a simplicial opfibration.
	The \fm{fiberwise opposite} of $P$ is the simplicial opfibration $P_\op \colon \E_\op \to \D$ given by 
	\[
		\Gr \circ  \op_s \circ \Gr^{-1}(P).
	\]
	Note that we have deliberately avoided writing $P^\op$ and $\E^\op$, since these mean the direct application of $\op_s$ to $P$ and $\E$, which is not what we want.
\end{definition}

\begin{corollary}
	Let $\C$ be a strict monoidal simplicial category. Then
	\[
		(\C^\otimes)_\op \cong (\C^\op)^\otimes.
	\]
\end{corollary}
\begin{proof}
	Apply $\Gr$ to Lemma \ref{lem:C-bullet-op}, and note that $\Gr\, (\C^\bullet)^\op \cong (\C^\otimes)_\op$.
\end{proof}

\subsection{Opposites of $\infty$-categories}

We now turn to opposites of simplicial sets and quasicategories, and relate these to opposites of simplicial categories.
In this and the next subsection, we will make frequent use of the notation and results of \ref{sec:models} and \ref{sec:st-un}, so the reader is encouraged to review them before proceeding.

To avoid unnecessary complexity in our exposition and proofs, we will freely use the fact that $\Delta$, the simplex category, is equivalent to the category $\cat{floSet}$ of finite, linearly ordered sets and order preserving functions between them. In fact, $\Delta$ is a \textit{skeleton} of $\cat{floSet}$, so the equivalence is given by the inclusion $\Delta\hookrightarrow\cat{floSet}$.

\begin{definition}
	Define the functor $\rev\colon\Delta\to \Delta$ to be the functor that takes a finite linearly ordered set to the same set with the reverse ordering. Then given $X\in \sSet = \mathsf{Fun}(\Delta,\cat{Set})$, we define $\op_\Delta X$ to be the simplicial set $X\circ \rev$. This defines a functor $\op_\Delta\colon \sSet\to \sSet$. We will often write $X^\op$ instead of $\op_\Delta X$. 
\end{definition}

\begin{definition}
	Define the functor $\op_\Delta^+\colon \sSet^+\to \sSet^+$ to be the functor that takes a marked simplicial set $(X,W)$ to $(\op_\Delta X,W)$, where we use the fact that there is a bijection between the 1-simplices of $\op_\Delta X$ and those of $X$.
\end{definition}

\begin{lemma}\label{lem:opselfadj}
	The functors $\op_\Delta$ and $\op_\Delta^+$ are self-adjoint.
\end{lemma}

\begin{lemma}
	If $X$ is a quasicategory, then so is $X^\op$.
\end{lemma}

The functors $\op_s, \op_\Delta$ and $\op_\Delta^+$ are related in the following manner:

\begin{lemma}\label{lem:opcommute}
The following diagram commutes:
\begin{center}
	\begin{tikzcd}[sep = large]
		\sCat^\circ \ar[r,"\N"]\ar[d,"\op_s"'] & \sSet^\circ\ar[d,"\op_\Delta" description]\ar[r,"\natural"] & \sSet^+\ar[d, "\op_\Delta^+"]\\
		\sCat^\circ \ar[r,"\N"]& \sSet^\circ\ar[r,"\natural"] & \sSet^+ 
	\end{tikzcd}
\end{center}
\end{lemma}
\begin{proof}
	The right hand square of the above diagram obviously commutes, so it only remains to show that $\N\circ \op_s\cong \op_\Delta\circ \N$.
	Recall that the nerve of a simplicial category $\C$ is the simplicial set determined by the formula 
	\[
		\Hom_{\sSet}(\Delta^n,\N \C) = \Hom_{\sCat}(\CC[\Delta^n],\C)
	\] 
	where $\CC[\Delta^n]$ is the value of the functor $\CC\colon \Delta\to \sCat$ defined in \cite{htt}*{1.1.5.1, 1.1.5.3} at the finite linearly ordered set $\{0<1<\cdots<n\}$. 
	Moreover, by extending along the Yoneda embedding $\Delta\to \sSet$, we obtain (cf.~the discussion following Example 1.1.5.8 of \cite{htt}) a colimit preserving functor $\CC\colon\sSet\to \sCat$ which is left adjoint to $\N$. This justifies using the notation $\CC[\Delta^n]$ for the application of $\CC$ to $\{0<1<\cdots<n\}$. 
	It is not hard to check from definitions that, for any finite linearly ordered set $I$ the simplicial categories $\CC[I]^\op$ and $\CC[I^\op]$ are equal and that this identification is natural with respect to the morphisms of $\Delta$.  
	So by using this fact, the fact that $\CC\dashv\N$, and liberally applying the self-adjointness of $\op_s, \op_\Delta$ and $\op_\Delta^+$ (Lemmas \ref{lem:op-s-selfadj} and \ref{lem:opselfadj}), we have the following sequence of isomorphisms:
	
	\begin{align*}
	\Hom_{\sSet}(\Delta^n,\N(\C)^\op)&\cong \Hom_{\sSet}((\Delta^n)^\op,\N(\C))\\
	&\cong \Hom_{\sCat}(\CC[(\Delta^n)^\op],\C)\\
	&\cong \Hom_{\sCat}(\CC[\Delta^{n}]^\op,\C)\\
	&\cong \Hom_{\sCat}(\CC[\Delta^n],\C^\op)\\
	&\cong \Hom_{\sSet}(\Delta^n,\N(\C^\op)).
	\end{align*}
	
	\noindent All of our constructions are natural with respect to the morphisms of $\Delta$, so we have the result.
\end{proof}
%

%
%

\begin{corollary}\label{cor:F-f-op}
	Let $F \colon \D \to \sCat$ be a functor such that each $Fd$ is fibrant, and let $f = \N F$.
	Then 
	\[
		f^\op = (\N F)^\op \cong \N(F^\op).
	\]
\end{corollary}

\begin{corollary}
	Let $f \colon \D \to \sSet$ be a functor such that each $fd$ is a quasicategory.
	Then
	\[
		(f^\op)^\natural = (f^\natural)^\op.
	\]
\end{corollary}

The preceding Corollary is about functors $\D \to \sSet$ taking values in quasicategories.
Taking the nerve of such a functor, we obtain a \emph{vertex} in the quasicategory $(\Cat_\infty)^{\N(\D)}$. 
From now on, we restrict ourselves to the quasicategories $\Cat_\infty, (\Cat_\infty)^{\N(\D)}$ and $\coCart_{/\N(\D)}$, so that all future statements are about \emph{vertices} in these quasicategories.

By \cite[Theorem 7.2]{barwickschommerpries}, there is a unique-up-to-homotopy non-identity involution of the quasicategory $\Cat_\infty$, as it is a theory of $(\infty,1)$-categories. 
Thus, this involution, which we denote $\op_{\infty}$, must be equivalent to the nerve of $\op_\Delta^+$. So we have the following lemma:

\begin{lemma}
	Let $\op_\infty \colon \Cat_\infty \to \Cat_\infty$ denote the above involution on $\Cat_\infty$.
	Then $\op_\infty \simeq \N(\op_\Delta^+)$.
\end{lemma}

\begin{corollary}\label{cor:markedop}
	Let $f \colon \D \to \sSet$ be a functor such that each $fd$ is a quasicategory, and continue to write $f$ for $f^\natural \colon \D \to \sSet^+$.
	In the quasicategory $(\Cat_\infty)^\D$, we have an equivalence 
	\[
		\N(f^\op) \simeq \N(f)^\op,
	\]
	where $f^\op = \op_\Delta^+ \circ f$ and $\N(f)^\op = \op_\infty \circ \N(f)$.
\end{corollary}

\begin{proof}
	By the functoriality of the (large) simplicial nerve functor and the previous Lemma, we have $\N(f^\op) \simeq \N(op_\Delta^+)\circ \N(f) \simeq \op_\infty \circ \N(f)$. 
\end{proof}

\subsection{Opposites of fibrations and monoidal quasicategories}

We now define fiberwise opposites of a coCartesian fibration, in a manner similar to Definition \ref{def:sfib-op}, keeping in mind that we need to work within the quasicategory $\coCart_{/S}$.

\begin{definition} \label{def:cocartfib-op}
	Let $p \colon X \to S$ be a coCartesian fibration of quasicategories, treated as a vertex of $\coCart_{/S}$.
	The \fm{fiberwise opposite} of $p$ is the coCartesian fibration corresponding to the vertex 
	\[
		\Gr_\infty \circ \op_\infty \circ \Gr_\infty^{-1} (p) \in \coCart_{/S}.
	\]
	Denote this coCartesian fibration by $p_\op \colon X_\op \to S$.
	(Again, we do not write $p^\op$ or $X^\op$, since these refer to the direct application of $\op_\Delta^{+}$).
\end{definition}

\begin{theorem} \label{thm:F-op-commute}
	Let $F \colon \D \to \sCat$ be a functor such that each $Fd$ is fibrant.
	In the quasicategory $\coCart_{/\N(\D)}$, there is an equivalence of vertices 
	\[
		\N \Gr (F^\op) \simeq \N\Gr \, (F)_\op,
	\]
	i.e.~the following diagram commutes on objects, and up to equivalence in $\coCart_{/\N(\D)}$.
	\[
		\begin{tikzcd}[row sep = large]
			\sCat^{\D} \ar[r, "\Gr"] \ar[d, "\op" description] & \opFib_{/\D} \ar[r, "\N"]  
			& \coCart_{/\N(\D)} \ar[d, "\op" description]
			\\
			\sCat^{\D} \ar[r, "\Gr"] & \opFib_{/\D} \ar[r, "\N"] & \coCart_{/\N(\D)}		
		\end{tikzcd}
	\]
\end{theorem}
\begin{proof} 
	We have a string of equivalences:	
	\begin{align*}
	\N \Gr\, (F)_\op &=  { \Gr_\infty \circ \op_\infty \circ \Gr_\infty^{-1} }(\N\Gr \, (F)) & & \text{(Definition \ref{def:cocartfib-op})} \\
	&\simeq \Gr_\infty \circ \op_\infty \circ \Gr_\infty^{-1} \Gr_\infty \N (f) & & \text{(Corollary \ref{cor:gr-infty-gr})}\\ 
	&\simeq \Gr_\infty \circ \op_\infty \circ \N(f) & &\text{(Definition \ref{def:grinfinity})} \\
	&\simeq \Gr_\infty \N( {f^\op}) & & \text{(Corollary \ref{cor:markedop})} \\
	&\simeq \N\Gr (F^\op) & & \text{(Corollary \ref{cor:gr-infty-gr})}
	\end{align*}
	where $f = \N F$ and $f^\op \cong \N(F^\op)$ by Corollary \ref{cor:F-f-op}.
\end{proof}

\begin{remark}
	The reader following the above proof closely should be aware of the fact that we implicitly use  Proposition \ref{prop:rel-nerve-infty-gr} \cite{htt}*{3.2.5.21} several times. 
\end{remark}

Finally, we turn our attention back to monoidal quasicategories and monoidal simplicial categories.

\begin{lemma}
	Let $p \colon X \to \N(\Delta^\op)$ define a monoidal structure on $X_{[1]}$.
	Then $p_\op \colon X_\op \to \N(\Delta^\op)$ defines a monoidal structure on $(X_{[1]})^\op$.
\end{lemma}
\begin{proof}
	It is easy to check that the coCartesian fibration $p_\op$ is a monoidal quasicategory, and that $(X_\op)_{[1]} \simeq (X_{[1]})^\op$.
\end{proof}

\begin{theorem}\label{thm:opcommute}
	Let $\C$ be a strict monoidal fibrant simplicial category and equip $\C^\op$ with its canonical monoidal structure.
	Then $\N^\otimes(\C^\op)$ and $\N^\otimes (\C)_\op$ define equivalent monoidal structures on $\N(\C^\op) \simeq \N(\C)^\op$. 
\end{theorem}
\begin{proof}
	Combine Lemma \ref{lem:C-bullet-op} with Theorem \ref{thm:F-op-commute}, taking $F = \C^\bullet$.
\end{proof}

\addcontentsline{toc}{section}{References}
\bibliography{references}

\appendix
\section{Appendices}

\subsection{Models for $\infty$-categories, and their nerves} \label{sec:models}

In this paper, we pass between simplicially enriched categories, $\sCat$, and simplicial sets, $\sSet$. We also often invoke \textit{marked} simplicial sets $\sSet^+$. In this section, we describe how these categories, equipped with suitable model structures, serve as models for a category of $\infty$-categories, and how they are related.

\begin{definition} We recall the definitions of the three categories above with certain model category structures:
	\begin{enumerate}
		\item Let $\sCat$ denote the category of simplicially enriched categories in the sense of \cite{kellyenriched}, with the \textit{Bergner} model structure described in \cite{bergmodelcat}. In particular, the fibrant objects are the categories enriched in Kan complexes and the weak equivalences are the so-called Dwyer-Kan (or DK) equivalences of simplicial categories.
		
		\item Let $\sSet$ denote the category of simplicial sets with the \textit{Joyal} model structure as described in \cite{joyalapps} and \cite{htt}. The fibrant objects are the quasicategories, and the weak equivalences are the categorical equivalences of simplicial sets.
		
		\item Let $\sSet^+$ denote the category of \emph{marked simplicial sets}. Its objects are pairs $(S,W)$ where $S$ is a simplicial set and $W$ is a subset of $S[1]$, the collection of 1-simplices of $S$. The model structure on $\sSet^+$ is given by \cite{htt}*{3.1.3.7}. By \cite{htt}*{3.1.4.1}, the fibrant objects are the pairs $(S,W)$ for which $S$ is a quasicategory and $W$ is the set of 1-simplices of $S$ that become isomorphisms after passing to the homotopy category (i.e.~the equivalences of $S$). The weak equivalences, by \cite{htt}*{3.1.3.5}, are precisely the morphisms whose underlying maps of simplicial sets are categorical equivalences.
	
		\item Let $\RelCat$ denote the category of \emph{relative categories}, whose objects are pairs $(\cat{C},\cat{W})$, where $\cat{C}$ is a category and $\cat{W}$ is a subcategory of $\cat{C}$ that contains all the objects of $\cat{C}$. In \cite{barwickkanrelcats}, it is shown that $\RelCat$ admits a model structure, but we will not need it here. We only point out that any model category $\cat{C}$ has an underlying relative category in which $\cat{W}$ is the subcategory containing every object of $\cat{C}$ with only the weak equivalences as morphisms. 
	\end{enumerate}
\end{definition}

\begin{definition}
	Given a model category $\cat{C}$, we will denote by $\cat{C}^\circ$ the full subcategory spanned by bifibrant (i.e\ fibrant and cofibrant) objects.
\end{definition}

\begin{definition}
	We also introduce several functors which are useful in comparing the above categories as models of $\infty$-categories:
	\begin{enumerate}
		\item Let $\N\colon\sCat\to \sSet$ be the \textit{simplicial nerve} functor (first defined by Cordier)  of \cite{htt}*{1.1.5.5}. Crucially, if $\C$ is a fibrant simplicial category, then $\N\C$ is a quasicategory. This nerve has a left adjoint $\CC$. 
		\[
			\begin{tikzcd}[column  sep = large]
				\sSet \ar[r, bend left = 20, "\CC", ""{name = L}]
				&
				\sCat \ar[l, bend left = 20, "\N", ""{name = R}] 	
				\ar[from = L, to = R, symbol = \dashv]		
			\end{tikzcd}
		\]		
		\item Let $L^H\colon\RelCat\to \sCat$ denote the \textit{hammock localization} functor, defined in \cite{dwyerkancalculating}. 
		\item  Let $(-)^\natural\colon\sSet^\circ\to \sSet^+$ denote the functor, defined in \cite{htt}*{3.1.1.9\footnote{This refers to the published version listed in our references. The same definition appears at 3.1.1.8 in the April 2017 version on Lurie's website.}}, that takes a quasicategory $C$ to the pair $(C,W)$ where $W$ is the collection of weak equivalences\footnote{We are using the fact that the unique map $p \colon C \to \Delta^0$ is a Cartesian fibration iff $C$ is a quasicategory, and the $p$-Cartesian edges are precisely the weak equivalences.} in $C$.	
		\item Let $(-)^{\sharp}	\colon\sSet\to \sSet^+$ denote the functor, defined in \cite{htt}*{3.1.0.2} that takes a simplicial set $S$ to the pair $(S,S[1])$, in which every edge of $S$ has been marked.
		\item Let $\uq\colon \RelCat\to \sSet$ denote the \emph{underlying quasicategory} functor of \cite{mazelgeeadjunctions}, given by the composition
		\[
			\RelCat \xrightarrow{\;L^H\;} \sCat \xrightarrow{\;\RR\;} \sCat \xrightarrow{\;\N\;} \sSet
		\]
		 where $\RR\colon \sCat\to \sCat$ is the fibrant replacement functor  of simplicial categories defined in \cite{mazelgeeadjunctions}*{\S 1.2}. Note that, because of the fibrant replacement, $\uq(\cat{C},\cat{W})$ is indeed a quasicategory for any relative category $(\cat{C},\cat{W})$.  
	\end{enumerate}
\end{definition}

\noindent We can now give a definition of \textit{the} quasicategory of $\infty$-categories:

\begin{definition} \label{def:cat-infty}
	Since the fibrant-cofibrant objects in $\sSet^+$ correspond to quasicategories, we let the \fm{quasicategory of quasicategories}, or of $\infty$-categories, be:
	\[
	\Cat_\infty := \N (\sSet^+)^\circ,
	\]
	where we write $\N (\sSet^+)^\circ$ instead of the more cumbersome $\N\big( (\sSet^+)^\circ \big)$.
\end{definition}

\begin{remark} \label{rem:nerve-notation}
	Going forward, we will often write  $\N (-)^\circ$ instead of $\N\big((-)^\circ \big)$ to indicate the simplicial nerve applied to the bifibrant subcategory of a simplicial model category.
\end{remark}

\begin{theorem}
	The underlying quasicategories of the model categories $\sCat$, $\sSet$ and $\sSet^+$are all equivalent to $\Cat_\infty$. 
\end{theorem}

\begin{proof}
	First note that \cite{hinichdwyerkan}*{Proposition 1.5.1} implies that a Quillen equivalence of model categories induces an equivalence of underlying quasicategories. There are Quillen equivalences $\sCat\leftrightarrows\sSet$ \cite{bergner}*{Theorem 7.8} and $\sSet\leftrightarrows\sSet^+$ \cite{htt}*{3.1.5.1 (A0)}. As a result, there are equivalences of quasicategories $\uq(\sSet,\var{WE})\to \uq(\sSet^+,\cat{WE})$, where $\cat{WE}$ denotes the collection of weak equivalences between marked simplicial sets, and $\uq(\sCat,\cat{DK})\to \uq(\sSet,\cat{WE})$, where $\cat{DK}$ denotes the collection of Dwyer-Kan equivalences. It then follows, by \cite{htt}*{3.1.3.5}, that there are equivalences of marked simplicial sets $\uq(\sSet,\cat{WE})^\natural\leftarrow \uq(\sSet^+,\cat{WE})^\natural$ and $\uq(\sCat,\cat{DK})^\natural\to \uq(\sSet,\cat{WE})^\natural$. 
	
	Now by \cite{hinichdwyerkan}*{Proposition 1.4.3} and its corollary, we have a (Dwyer-Kan) equivalence of simplicial categories $(\sSet^+)^\circ\to L^H(\sSet^+,\cat{WE})$. By definition of fibrant replacement, we also have equivalences $(\sSet^+)^\circ\to \RR(\sSet^+)^\circ$. Since the latter morphism is between fibrant objects, and the right Quillen adjoint $\N$ preserves equivalences between fibrant objects (by Ken Brown's Lemma), we have an equivalence of simplicial sets $\N(\sSet^+)^\circ\to \uq (\sSet^+,\cat{WE})$. Thus another application of \cite{htt}*{3.1.3.5} gives an equivalence of marked simplicial sets $(\N(\sSet^+)^\circ)^\natural\to \uq (\sSet^+,\cat{WE})^\natural$.
	
	So we have equivalences of marked simplicial sets:
	\[
	(\N(\sSet^+)^\circ)^\natural\to \uq (\sSet^+,\cat{WE})^\natural\to \uq(\sSet,\cat{WE})^\natural \to \uq(\sCat,\cat{DK})^\natural 
	\]
	
	\noindent These imply the result after applying the (large) nerve to the (large) quasicategory of marked simplicial sets. 	 
\end{proof}

\subsection{Straightening, unstraightening and $\Gr_\infty$}
	\label{sec:st-un}
	This section is a summary of results from \cite{htt}*{3.2 and 3.3} regarding straightening and unstraightening.

	\begin{theorem}[\cite{htt}*{3.2.0.1}]
		Let $S$ be a simplicial set, $\D$ a simplicial category, and $\phi \colon \CC[S] \xto{\; \simeq \;} \D$ an equivalence of simplicial categories.
		Then there is a Quillen equivalence
		\[
			\begin{tikzcd}[column  sep = large]
				{}_{\phantom{S}} (\sSet^+)^\D \ar[r, bend left = 20, "\Un_\phi^+", ""{name = L}, start anchor = north east, end anchor = north west]
				&
				 (\sSet^+)_{/S} \ar[l, bend left = 20, "\St_\phi^+", ""{name = R}, start anchor = south west, end anchor = south east] 	
				\ar[from = L, to = R, symbol = \vdash]		
			\end{tikzcd}
		\]
		where $(\sSet^+)_{/S}$ is the category of marked simplicial sets over $S$ with the coCartesian model structure, and $(\sSet^+)^\D$ is the category of $\D$ shaped diagrams in marked simplicial sets with the projective model structure.
	\end{theorem}

	\begin{lemma}[\cite{htt}*{3.2.4.1}]
		Both $(\sSet^+)_{/S}$ and $(\sSet^+)^\D$ are simplicial model categories, and $\Un_\phi^+$ is a simplicial functor\footnote{But $\St_\phi^+$ is not always a simplicial functor.} which induces an equivalence of simplicial categories
		\[
		(\Un_\phi^+)^\circ\colon \big((\sSet^+)^\D \big)^\circ\xrightarrow{\;\simeq \;} \big((\sSet^+)_{/S}\big)^\circ.
		\]
	\end{lemma}

	\begin{corollary}[\cite{htt}*{A.3.1.12}] 
	\label{cor:st-un-quasi}
		Taking the nerve of this equivalence, there is an equivalence of quasicategories\footnote{We use the notational convention in Remark \ref{rem:nerve-notation}.}
		\[
		\N(\Un_\phi^+)^\circ \colon \N\big((\sSet^+)^\D \big)^\circ \xrightarrow{\; \simeq \;} \N\big((\sSet^+)_{/S}\big)^\circ.
		\]
	\end{corollary}

\begin{remark}
	Note that, for \cite{htt}*{A.3.1.12} to apply above, it is essential that all of the objects of $(\sSet^+)_{/S}$ are cofibrant. This follows from \cite{htt}*{3.1.3.7} when we set $S=\Delta^0$ and the recollection that every object of $\sSet$ is cofibrant in Joyal model structure.
\end{remark}

	By \cite{htt}*{3.1.1.11\footnote{This is 3.1.1.10 in the April 2017 version on Lurie's website.}}, the vertices of $\N \big((\sSet^+)_{/S}\big)^\circ$ are precisely maps of marked simplicial sets of the form $X^\natural \to S^\sharp$ where $X \to S$ is a coCartesian fibration.
	We may thus \emph{identify} $X \to S$ with $X^\natural \to S^\sharp$ and treat the vertices of $\N \big((\sSet^+)_{/S}\big)^\circ$ as coCartesian fibrations over $S$.
	This motivates and justifies the following notation:

	\begin{definition}\label{def:cocartqcat}
		The \fm{quasicategory of coCartesian fibrations over $S$} is
		\[
			\coCart_{/S} := \N \big((\sSet^+)_{/S}\big)^\circ.
		\] 
	\end{definition}

	\begin{corollary} \label{cor:gr-corr-infty}
		There is an equivalence of quasicategories
		\[
			(\Cat_\infty)^S \simeq \coCart_{/S}.
		\]		
	\end{corollary}
	\begin{proof}
	By Corollary \ref{cor:st-un-quasi} with $\D = \CC[S]$ and $\phi$ the identity, it suffices to show that we have an equivalence of quasicategories
	\[
		\N\big((\sSet^+)^{\CC[S]} \big)^\circ \simeq (\Cat_\infty)^S.
	\]
	But this is precisely \cite{htt}*{4.2.4.4}, which states that
	\[
		\N\big((\sSet^+)^{\CC[S]} \big)^\circ \simeq \big(\N(\sSet^+)^\circ \big)^S,
	\]
	together with Definition \ref{def:cat-infty}.
	\end{proof}
	
%

	\begin{definition}\label{def:grinfinity}
		Let $\Gr_\infty$ denote the above equivalence of quasicategories,
		\[
			\begin{tikzcd}
				 (\Cat_\infty)^S \ar[rr, bend left = 15, "\Gr_\infty"] & \simeq & \coCart_{/S} \ar[ll, bend left = 15, "\Gr_\infty^{-1}"]
			\end{tikzcd}
		\]
		and let $\Gr_\infty^{-1}$ denote its weak inverse (i.e.\ there are natural equivalences of functors $\sf{Id}_{\coCart_{/S}}\simeq\Gr_\infty\circ \Gr_\infty^{-1}$ and $\sf{Id}_{(\Cat_\infty)^S}\simeq \Gr_\infty^{-1}\circ \Gr_\infty$). 
	\end{definition}

	\begin{remark}
		The existence of a weak inverse $\Gr_\infty^{-1}$ is a result of the ``fundamental theorem of quasicategory theory'' \cite{rezkstuff}*{\S 30}.
		By \cite{htt}*{5.2.2.8}, one can check that $\Gr_\infty$ and $\Gr_\infty^{-1}$ are adjoints in the sense of \cite{htt}*{5.2.2.1}, but we will not need that here.

		Note that $\Gr_\infty^{-1}$ is \emph{not} the nerve of $(\St_\phi^+)^\circ$ (the latter is not even a simplicial functor).
		See \cite{riehl2017comprehension}*{6.1.13, 6.1.22} for a description of $\Gr_\infty^{-1}$ \emph{on objects}, and \cite{riehl2017comprehension}*{6.1.19} for an alternative description of $\Gr_\infty$.
	\end{remark}

	\begin{definition}[\cite{htt}*{3.3.2.2}]
	\label{def:classified}
		For $p \colon X \to S$ a coCartesian fibration, a map $f \colon S \to \Cat_\infty$ \fm{classifies $p$} if there is an equivalence of coCartesian fibrations $X \simeq \Gr_\infty f$.
	\end{definition}

\subsection{Functors out of $\CC[\Delta^n]$}\label{sec:func-cc}
	We review the characterization of simplicial functors out of $\CC[\Delta^n]$ that will be used in the proof of Theorem \ref{thm:gr-rel-nerve}.
	All material here is from \cite{riehl2017comprehension}, with some slight modifications in notation and terminology.

	Throughout, $[n]$ denotes the poset $\{0 < 1 < \dots < n\}$.

	\begin{definition}[\cite{riehl2017comprehension}*{4.4.6}]
	\label{def:bead}
		Let $I = \{i_0 < i_1 < \dots < i_m \}$ be a subset of $[n]$ containing at least $2$ elements (i.e.\ $m \geq 1$).

		An \fm{$r$-dimensional bead shape} of $I$, denoted $\langle I_0 | I_1 | \dots | I_r \rangle$, is a partition of $I$ into non-empty subsets $I_0,\dots, I_r$ such that $I_0 = \{i_0, i_m\}$.
	\end{definition}

	\begin{example} \label{eg:bead}
		A $2$-dimensional bead shape of $I = \{0,1,2,3,5,6\}$:
		\begin{align*}
			I_0 &= \{0,6\} ,
			& I_1 &= \{3\} ,
			& I_2 &= \{1,2,5\}.
		\end{align*}
		We write $S_{\langle I_0 | I_1 | I_2 \rangle}$ to mean the same thing as $S_{\langle 06|3|125\rangle}$.
	\end{example}

	\begin{lemma}[\cite{riehl2017comprehension}*{4.4.9}]
	\label{lem:simplicial-functor}
		A simplicial functor $S \colon \CC[\Delta^n] \to \K$ is precisely the data of:
		\begin{itemize}
			\item For each $i \in [n]$, an object $S_i \in \K$
			\item For each subset $I = \{i_0 < \dots < i_m\} \subseteq [n]$ where $m \geq 1$, and each $r$-dimensional bead shape $\bead{I_0 | \dots | I_r}$ of $I$, an $r$-simplex $S_{\langle I_0 | \dots | I_r \rangle}$ in $\K(S_{i_0}, S_{i_m})$ whose boundary is compatible with lower-dimensional data.
		\end{itemize}
	\end{lemma}
	The main benefit of this description is that \emph{no further coherence conditions} need to be checked.
	Instead of describing what it means for the boundary to compatible with lower-dimensional data, which can be found in \cite{riehl2017comprehension}, we illustrate this with an example.
	But first, we introduce the abbreviation
	\[
		S_\bead{i_0i_1 \dots i_m} := S_\bead{i_{m-1}i_m}S_\bead{i_{m-2}i_{m-1}}\dots S_\bead{i_1 i_2} S_\bead{i_0 i_1}.
	\]

	\begin{example}
		The bead shape in Example \ref{eg:bead} is $2$-dimensional, so $S_\bead{I_0|I_1|I_2} = S_{\langle 06 | 3 | 125  \rangle}$ should be a $2$-simplex in $\K(S_0, S_6)$.
		The boundary of this $2$-simplex is compatible with lower-dimensional data in the sense that it is given by the following:
		\begin{itemize}
			\item The first vertex is always $S_\bead{I_0}$, which in this case is $S_\bead{06} \in \K(S_0, S_6)_0$.
			\item The last vertex is always $S_\bead{I}$, which in this case is $S_\bead{012356}$.
			Between the first and last vertex, we have
			\[
				S_\bead{06} \xrightarrow{\quad S_\bead{06|1235} \quad} S_\bead{012356} \quad \quad \in \K(S_0,S_6)_1,
			\]
			representing the insertion of $I_1 \cup I_2 \cup \dots \cup I_r$ into $I_0$.
			This is always the starting edge of $S_\bead{I_0|\dots|I_r}$.
			\item The remaining vertices and edges are generated by first inserting $I_1$ into $I_0$, then $I_2$ into $I_0 \cup I_1$ and so on, up to inserting $I_r$ into $I\setminus I_r$.
			\item In our case, we first insert $I_1 = \{3\}$ into $I_0$. 
			This yields the vertex $S_\bead{I_0 \cup I_1} =S_\bead{036} = S_\bead{36}S_\bead{03}$ and the edge
			\[
				S_\bead{06} \xrightarrow{\quad S_\bead{06|3} \quad} S_\bead{036} \quad \quad \in \K(S_0,S_6)_1.
			\]
			\item Next, we insert $I_2 = \{1,2,5\}$ into $I_0 \cup I_1$.
			Since this gives all of $I$ and we already have $S_\bead{I}$, we do not need to add any more vertices.
			We only add the edge
			\[
				S_\bead{036} \xrightarrow{\quad S_\bead{36|5} S_\bead{03|12} \quad} S_\bead{01235}\quad \quad \in \K(S_0,S_6)_1,
			\]
			where $S_\bead{36|5} \in \K(S_3,S_5)_1$ and $S_\bead{03|12} \in \K(S_0, S_3)_1$.
			Note that $5$, lying between $3$ and $6$, goes into $S_\bead{36}$, as indicated by $S_\bead{36|5}$; similarly, $1$ and $2$ go into $S_\bead{03}$, as indicated by $S_\bead{03|12}$. 
			We denote this composite
			\[
				S_\bead{036|125} := S_\bead{36|5} S_\bead{03|12}.
			\]
			\item We can then choose $S_\bead{06|3|125}$ to be \emph{any} $2$-simplex in $\K(S_0,S_6)$ fitting into the following:
			\[
				\begin{tikzcd}[row sep = huge]
					S_\bead{06} \ar[rr, "S_\bead{06|1235}", ""{name=U, below}] \ar[dr, "S_\bead{06|3}"'] & & S_\bead{012356} 
					\\
					& S_\bead{036} \ar[ur, "S_\bead{036|125}"'] \ar[from = U, Rightarrow, shorten >= 1em, "S_\bead{06|3|125}" near start]& 
				\end{tikzcd}
			\]
		\end{itemize}
	\end{example}

	\begin{remark}
		The rule that $I_0$ must have exactly $2$ elements in Definition \ref{def:bead} allows us to distinguish bead shapes from abbreviations. 
		For instance, $S_\bead{06|3}$ arises from a bead shape, while $S_\bead{036|125}$ is an abbreviation.

		Note that we \emph{should not} abbreviate the composite $S_\bead{036|125} S_\bead{06|3}$ as $S_\bead{06|1235}$, since the latter implies that we insert $\{1,2,3,5\}$ all at once into $\{0,6\}$. 
		Indeed, the point of $S_\bead{06|3|125}$ is to relate $S_\bead{036|125} S_\bead{06|3}$ and $S_\bead{06|1235}$.

		We only abbreviate $S_\bead{j_0 \dots j_\ell |\dots} S_\bead{i_0 \dots i_k|\dots}$ as $S_\bead{i_0 \dots i_k j_1 \dots j_\ell | \dots}$ if $i_k = j_0$.
		The upshot is that \emph{there is an entirely unambiguous process} of converting an abbreviation into a composite of bead shapes, and \emph{not all compsites} of bead shapes may be abbreviated.
		See \cite{riehl2017comprehension}*{4.2.4} for details.
	\end{remark}

\end{document}